\documentclass[12pt,a4paper]{article}

\usepackage[margin=1in]{geometry}
\usepackage{amsmath,amsfonts,amssymb}
\usepackage{hyperref}
\usepackage[all]{xy}
\usepackage[english]{babel}
\usepackage{amsthm}
\usepackage{faktor}

 \newtheorem{thm}{Theorem}[section]
 \newtheorem{lem}[thm]{Lemma}
 \newtheorem{rem}[thm]{Remark}

\newcommand{\Map}{\mathtt{Map}}
\newcommand{\Sym}{\mathtt{Sym}}
\newcommand{\BiMod}{\mathbf{BiMod}}
\newcommand{\Dun}{\mathbf{D}_1}
\newcommand{\Dzun}{\mathbf{D}_{0,1}}
\newcommand{\Assoc}{\mathbf{As}}
\newcommand{\RMod}{\mathbf{RMod}}
\newcommand{\LMod}{\mathbf{LMod}}
\newcommand{\Cpx}{\mathbf{Cpx}}
 
\title{Not too little intervals for quantum mechanics}
\author{Damien Calaque\thanks{IMAG, Univ Montpellier, CNRS, Montpellier, France \\ \indent\href{mailto:damien.calaque@umontpellier.fr}{damien.calaque@umontpellier.fr}}}
\date{August 2024}

\begin{document}

\maketitle

\begin{abstract}
This short paper illustrates the general framework introduced in the paper ``\emph{Not too little discs}'', joint with Victor Carmona, on yet another one dimensional example. It exhibits a discrete model for the free scalar field on the real line, adapting the treatment from the book of Costello--Gwilliam to the discrete setting. \end{abstract}

\setcounter{tocdepth}{1}
\tableofcontents

\section{Introduction}

This is a companion paper to \cite{CC}, where Victor Carmona and the author introduce a general method for constructing algebras over the little discs operad from discrete models. 
We provide a detailed study of a one dimensional example that can be seen as a discrete version of the free scalar field on the real line. Our treatment is an adaptation of \cite[Ch.~4 Sect.~3]{CG1} to the discrete setting. 

\paragraph{Factorization algebras in quantum field theory}

In \cite{CG1,CG2}, Costello and Gwilliam developped a formalism for quantum field theory where the notion of a \emph{factorization algebra} plays a central role, as encoding both the algebraic and local-to-global structure of quantum observables. 
Additionally, they made it compatible with (a version of) the Batalin--Vilkovisky formalism and renormalization techniques based on parametrices. 
In the case of \emph{topological} quantum field theories, factorization algebras are locally constant and thus reduce to algebras over the little disc operad. 

\paragraph{Not too little discs}

In \cite{CC}, an approach for constructing topological quantum field theories on flat space without using parametrices, and based on discretization methods, is proposed. It relies on a new result (the main one of \cite{CC}) showing that local constantness above a given scale $R>0$ is enough to reconstruct an algebra over the little disc operad: one can then simply ignore the inaccurate information that comes from what happens at too small scale (compared the discretization mesh). As a proof of concept, \cite[Section 4]{CC} gives a one dimensional example where a discrete version of the first order formulation of topological quantum mechanics allows to quantize constant Poisson structures and to recover Weyl-type algebras. 

\paragraph{New example}

In this paper, we provide another one dimensional example (hence the name not too little \emph{intervals}), that is a discrete version of the second order formulation of topological quantum mechanics, also known as the free scalar field in $1d$. We discuss both the massless and massive cases, show that it also recovers the Weyl algebra, that the information about the mass appears in the translation symmetries, and that one can recover the Fock space using time reversal symmetry. 

\paragraph{Organization of the paper}

Section \ref{sec-nottolittle} is the most abstract one, recollecting everything one needs to know about the algebraic and operadic aspects of not too little intervals in order to deal with the example of interest in this paper. 
Subsection \ref{ssec-2.1} provides a quick summary of the main result of \cite{CC} in dimension $1$, saying that locally constant algebras over the colored operad of intervals of length $>2$ (shortly, locally constant $\Dun^1$-algebras) are associative algebras. 
Subsection \ref{ssec-2.2} discusses equivariance, under translations on the one hand, and under the time/orientation reversion automorphism of the real line on the other hand. 

In Section \ref{sec-massless} we construct a locally constant $\Dun^1$-algebra of quantum observables for the discrete free massless scalar field in $1d$. We recall the discrete Laplace operator in Subsection \ref{ssec-3.1}, compute the cohomology of its associated complex in Subsection \ref{ssec-3.2}, construct the locally constant not too little discs algebras of classical and quantum observables in Subsections \ref{ssec-3.3} and \ref{ssec-3.4}. For the quantum observables, we  show in Subsection \ref{ssec-3.5} that the corresponding associative algebra indeed coincides with the Weyl algebra, together with its natural time evolution  automorphism (Subsection \ref{ssec-3.6}) and its natural time reversal anti-involution (Subsection \ref{ssec-3.7}). 

We repeat the above in Section \ref{sec-massive} for the free \emph{massive} scalar field\footnote{In fact, Section \ref{sec-massive} also encompasses the massless case. But, for exposition purpose and clarity, we think it is helpful to first discuss the massless case in Section \ref{sec-massless}. }. Classical observables are constructed in Subsection \ref{ssec-4.1} and their cohomology is computed in Subsection \ref{ssec-4.2}. Quantum observables are constructed in Subsection \ref{ssec-4.3}, where the identification with the Weyl algebra is also given. Equivariance under symmetries is discussed in Subsections \ref{ssec-4.4} (discrete time evolution) and \ref{ssec-4.5} (time reversal). As one could expect the value of the mass only ``matters'' in the expression of the (discrete) time evolution automorphism. 

\paragraph{Notation}

All along the paper, $\mathbb{K}$ is a field of characteristic zero. For any set $S$, we write 
\[
\Map(S,\mathbb{K}):=\mathbb{K}^S
\]
for the $\mathbb{K}$-vector space of $\mathbb{K}$-valued functions on $S$. If $S\subset T$, then the restriction morphism $\Map(T,\mathbb{K})\twoheadrightarrow \Map(S,\mathbb{K})$ has a canonical section $\Map(S,\mathbb{K})\hookrightarrow \Map(T,\mathbb{K})$, which is the extension by $0$; this allows to identify functions $S\to \mathbb{K}$ with functions $T\to\mathbb{K}$ having support in $S$ (we always implicitly do so without specific  notation or warning).  

The category of cochain complexes (with differential being of degree $+1$) of $\mathbb{K}$-vector spaces is denoted $\Cpx$. Results from the paper \cite{CC} are stated within the framework of $\infty$-categories and $\infty$-operads, and heavily rely on homotopy theory. This means that $\Cpx$ in fact denotes the $\infty$-category of cochain complexes, that is obtained from their genuine $1$-category by inverting quasi-isomorphisms. Nevertheless, all our examples carry \emph{strict} (i.e.~$1$-categorical) algebraic structures, and it is only when we discuss local constantness that we use a tiny bit of homotopy theory (we only request that certain maps are \emph{quasi-}isomorphisms, rather than plain isomorphisms, and we make everything very explicit). Therefore, the reader can safely ignore ``$\infty$-technicalities''. 

\section{Not too little intervals}\label{sec-nottolittle}


\subsection{Localization of not too little intervals and associative algebras}\label{ssec-2.1}

Let $R>0$, and let $D^R_1$ be the colored operad of bounded intervals in the real line, of length greater than $2R$. Colors of $\Dun^R$ are open intervals $(a,b)$, with $a,b\in\mathbb{R}$ such that $|b-a|>2R$, and sets of operations are defined as follows:
\[
\Dun^R(I_1,\dots, I_n;J)=\begin{cases} {* }&\text{if}~I_1\sqcup\cdots\sqcup I_n\subset J \\
\emptyset &\text{otherwise}
\end{cases}
\]
Symmetry group action and composition are obvious (note that to get a right action of $\Sigma_n$, a permutation $\sigma$ must send $I_1\sqcup\cdots\sqcup I_n\subset J$ to $I_{\sigma(1)}\sqcup\cdots\sqcup I_{\sigma(n)}\subset J$). 

There is a morphism $\gamma\colon \Dun^R\to \Assoc$, where $\Assoc$ is the operad encoding unital associative algebras. Recall that the operad $\Assoc$ has only one color and set of $n$-ary operations $\Assoc(n)=\Sigma_n$ the set of permutations on $n$ elements (i.e.~$\Assoc(n)$ is the right regular representation of $\Sigma_n$). The morphism $\gamma$ sends 
\begin{itemize}
    \item any open interval $I$ to the only color of $\Assoc$;
    \item a pairwise disjoint inclusion $I_1\sqcup\cdots \sqcup I_n\subset J$ to the unique permutation $\sigma$ such that \[
    I_{\sigma^{-1}(1)}<I_{\sigma^{-1}(2)}<\cdots <I_{\sigma^{-1}(n)}\,.
    \]
\end{itemize}

Notice that all unary operations in $\Dun^R$ are sent to identities, so that $\gamma$ factors through the operadic localization $\Dun^R[W^{-1}]$ of $\Dun^R$ at its subcategory $W=\Dun^R(1)$ of unary operations. 

\begin{rem}\label{rem-loc-strict-non-strict}
    There is a choice here: we could consider the strict ($1$-)operadic localization or the $\infty$-operadic one. We go for the $\infty$-operadic localization, but in dimension $1$ it doesn't matter so much as it is a direct consequence of the next result that they are actually equivalent as $\infty$-operads (see next Remark below). 
\end{rem}

\begin{thm}[{\cite[Theorem 2.4]{CC}}]
The induced morphism of $\infty$-operads $\Dun^R[W^{-1}]\to\Assoc$ is an equivalence. 
\end{thm}
\begin{rem}
    In \cite{CC} the result is proven for the little intervals operad $\mathbb{E}_1$, which is equivalent to $\Assoc$ \emph{via} the equivalence $\mathbb{E}_1\to \pi_0(\mathbb{E}_1)=\Assoc$. 
\end{rem}

The main interesting consequence (see \cite[Theorem 2.3]{CC}) of the above Theorem is that pulling-back along $\gamma$ provides an equivalence of $\infty$-categories 
between unital associative algebras and locally constant $\Dun^R$-algebra (in any symmetric monoidal 
$\infty$-category $\mathcal V$). Here a $\Dun^R$-algebra $A$ is called \emph{locally constant} if, for every abstract unary 
operation $I\subset J$ in $\Dun^R$, the corresponding operation morphism $A_I\to A_J$ is an equivalence (in  $\mathcal V$). In all our examples, equivalences will be quasi-isomorphisms (and $\mathcal V=\mathbf{Cpx}$). 

\medskip

In this paper we deal with an example of a locally constant $\Dun^1$-algebra in cochain complexes $\mathcal O^q$ whose 
cohomology is concentrated in degree $0$. In this case, the associative algebra $A$ one gets (i.e.~the associative algebra $A$ such that $\gamma^*A\simeq \mathcal O^q$) can be described in the following way: pick $a<b\leq c<d$ such that $|b-a|>2$ and $|d-c|>2$, define $A:=H^0\Big(\mathcal O^q\big((a,d)\big)\Big)$ and endow it with the following product $\star$: 
\[
\begin{matrix}
A & \tilde\longleftarrow & H^0\Big(\mathcal O^q\big((a,b)\big)\Big) &\\
\otimes& & \otimes & \underset{\mathrm{product}}{\overset{\mathrm{factorization}}{\relbar\joinrel\relbar\joinrel\relbar\joinrel\longrightarrow}} 
& H^0\Big(\mathcal O^q\big((a,d)\big)\Big)=A\,.\\
A& \tilde\longleftarrow & H^0\Big(\mathcal O^q\big((c,d)\big)\Big) & 
\end{matrix}
\]
Here the ``factorization product'' is the operation morphism coming from the abstract binary operation $(a,b)\sqcup(c,d)\subset(a,d)$ (the name is borrowed from factorization algebras). 

\subsection{Symmetries}\label{ssec-2.2}

\paragraph{Equivariance under translations}

The group $(\mathbb{R},+)$ acts on $\Dun^R$ by translations (note that we consider $\mathbb{R}$ as a discrete group): an element $r\in\mathbb{R}$ sends an interval $I$ to the interval $I+r$, and the operation $I_1\sqcup\cdots\sqcup I_n\subset J$ to the operation $(I_1+r)\sqcup\cdots\sqcup (I_n+r)\subset J+r$.  
For any subgroup $T\subset\mathbb{R}$ we can therefore perform the semi-direct product $\Dun^R\rtimes T$. 
The morphism $\gamma$ is invariant under this action, hence it gives rise to a morphism $\Dun^R\rtimes T\to\Assoc\rtimes T$ between semi-direct products, where $T$ acts trivially on $\Assoc$. 

\paragraph{Time reversal equivariance}

Let $C_2$ be the group with two elements, and denote by $\tau\in C_2$ the element that is not the identity. It acts on $\Dun^R$ in the following way: on colors $\tau\cdot I=-I$ and on operations
\[
\tau\cdot(I_1\sqcup\cdots\sqcup I_n\subset J)=\big((-I_1)\sqcup\cdots\sqcup (-I_n)\subset -J\big)\,.
\]
For obvious reasons, we call this action the \emph{time reversal action}. 
The group $C_2$ also acts on $\Assoc$: on $\Assoc(n)=\Sigma_n$ it acts by left multiplication by the permutation $(n\,n-1\dots 2\,1)$. We let the reader check that $\gamma$ is $C_2$-equivariant, inducing a morphism $\Dun^R\rtimes C_2\to \Assoc\rtimes C_2$ between semi-direct product operads. 

\medskip

A $C_2$-equivariant structure on an $\Assoc$-algebra $A$ (that is \emph{by definition} a lift of the $\Assoc$-algebra structure on $A$ to an $\Assoc\rtimes C_2$-algebra structure) is the data of an anti-involution $\tau:A\to A^{\mathrm{op}}$. One can extract a left $A$-module of coinvariants $A/\tau$ from this data, given as the quotient of $A$ by the left ideal generated by $a-\tau(a)$ for all $a\in A$.  
\begin{rem}
The $2$-colored operad $\LMod$, that encodes pairs $(A,M)$ of a unital associative algebra $A$ and a pointed left $A$-module $M$, can also be obtained by localizing a certain operad $\Dzun^R$ along a subcategory of its unary operations (we refer to \cite[\S3.2]{CC} for the details, where $\Dzun^R$ is defined as a version of $\Dun^R$ with boundary defect). It would be interesting to understand if the coinvariant module functor $\mathtt{Alg}_{\Assoc\rtimes C_2}\to \mathtt{Alg}_\LMod$ can be lifted to a functor $\mathtt{Alg}_{\Dun^R\rtimes C_2}\to \mathtt{Alg}_{\Dzun^R}$. 
\end{rem}

\begin{rem}
    Elaborating on the previous remark, one can observe that an algebra with a module is precisely the structure of observables of a one dimensional topological field theory with a boundary defect. The way we address such a defect using $C_2$-equivariance seems a bit \emph{ad hoc} and it would be interesting to rather have something along the lines of \cite{GRW,CEG}. 
\end{rem}

\section{Discrete free massless scalar field in $1d$}\label{sec-massless}

Let $a<b$ be real numbers. 

\subsection{The discrete Laplace operator and its associated $2$-term complex}\label{ssec-3.1}

Let $Q$ be the endomorphism of $\Map(\mathbb{Z},\mathbb{K})$ given by the standard discrete Laplace operator in $1d$: for every $f:\mathbb{Z}\to \mathbb{K}$, and every $x\in\mathbb{Z}$, 
\[
(Q f)(x):=f(x-1)-2f(x)+f(x+1)\,.
\]
Observe that if $f$ is supported in $(a,b)$ then $Qf$ is supported in $(a-1,b+1)$. 

We let $\mathbb{V}(a,b)$ be the cone of 
\[
\Map\big(\mathbb{Z}\cap(a+1,b-1),\mathbb{K}\big)\overset{Q}{\longrightarrow}\Map\big(\mathbb{Z}\cap(a,b),\mathbb{K}\big)\,.
\]

Note that the assignement $(a,b)\mapsto \mathbb{V}(a,b)$ naturally defines a $\Dun^R$-algebra with values in the symmetric monoidal category $(\Cpx,\oplus)$, for every $R>0$. Indeed, given and abstract operation $I_1\sqcup\cdots\sqcup I_n\subset J$ of $\Dun^R$, the associated operation morphism 
\[
\mathbb{V}(I_1)\oplus\cdots \oplus\mathbb{V}(I_n)\to \mathbb{V}(J)
\]
sends $(f_1,\dots,f_n)$ to $\sum_{j=1}^n f_j$. 

\subsection{Cohomology: local constantness}\label{ssec-3.2}

For every function $f:\mathbb{Z}\to\mathbb{K}$ having support in $[m,n]$, with $m,n\in\mathbb{Z}$, $(Qf)(m-1)=f(m)$ and $(Qf)(n+1)=f(n)$. 
Hence if $Qf=0$ then $f(m)=f(n)=0$, and thus $f$ has support in $[m+1,n-1]$. Hence, by induction, if $f$ has finite support and $Qf=0$, then $f=0$: in other words, $Q$ is injective on finitely supported maps. 
Therefore, the cohomology of $\mathbb{V}(a,b)$ is concentrated in degree $0$. 

\medskip

Moreover, if we assume that $b-a>2$, then the cohomology has dimension $2$ (by the rank theorem). In fact, the map 
\begin{equation}\label{eq:map-coho}
f\longmapsto \sum_{x\in\mathbb{Z}}f(x)\mathbf{q}
+
\sum_{x\in\mathbb{Z}}xf(x)\mathbf{p}
\end{equation}
provides an isomorphism between the cohomology of $\mathbb{V}(a,b)$ and $\mathbb{K}\mathbf{q}\oplus\mathbb{K}\mathbf{p}$, showing that $\mathbb{V}$ is a locally constant $\Dun^1$-algebra with values in $(\Cpx,\oplus)$. 
Actually, \eqref{eq:map-coho} gives a quasi-isomorphism $\mathbb{V}\,\tilde\to\, \gamma^*(\mathbb{K}\mathbf{q}\oplus\mathbb{K}\mathbf{p})$, where the $\Assoc$-algebra structure on $\mathbb{K}\mathbf{q}\oplus\mathbb{K}\mathbf{p}$ in $(\Cpx,\oplus)$ is the sum $+$. 

\begin{rem}[basis of cohomology]
Assuming $0\in(a,b)$, every $y\in\mathbb{Z}\cap(a,b-1)$ provides us with a section of \eqref{eq:map-coho}, given by $\mathbf{q}\mapsto\delta_0$, $\mathbf{p}\mapsto \delta_{y+1}-\delta_y$. 
\end{rem}

\subsection{Classical observables: the $1$-shifted symmetric pairing}\label{ssec-3.3}

Let us first introduce some piece of notation: we use overlined symbols $\overline{g}:\mathbb{Z}\to\mathbb{K} [1]$ to denote the suspension of functions
$g:\mathbb{Z}\to\mathbb{K}$, for which we use plain symbols. In particular, with this notation, the differential $d$ on $\mathbb{V}(a,b)$ is given by $d\overline{g}=Qg$. 

\medskip

We define 
\[
\langle\!\langle\overline{f},g\rangle\!\rangle=\sum_{x\in\mathbb{Z}}f(x)g(x)\,.
\]
This completely characterizes a degree $1$ symmetric pairing on $\mathbb{V}(a,b)$, that we also denote $\langle\!\langle\text{-,-}\rangle\!\rangle$. It is compatible with the differential (recall that $d\overline{f}=Q f$): 
\begin{eqnarray*}
\langle\!\langle d\overline{f},\overline{g}\rangle\!\rangle-\langle\!\langle \overline{f},d\overline{g}\rangle\!\rangle & = & 
\langle\!\langle Q f,\overline{g}\rangle\!\rangle-\langle\!\langle \overline{f},Q g\rangle\!\rangle \\
& = &
\sum_{x\in\mathbb{Z}}\Big(\big(f(x-1)-2f(x)+f(x+1)\big)g(x) \\&&-f(x)\big(g(x-1)-2g(x)+g(x+1)\big)\Big) 
\\
&=&0
\end{eqnarray*}

Moreover, it is also compatible with the $\Dun^1$-algebra structure. Therefore $\mathbb{V}$ is promoted to a locally constant $\Dun^1$-algebra with values in the category of complexes equipped with a degree $1$ symmetric pairing. 

\medskip

Now recall (see e.g.~\cite[\S4.1]{CC} and references therein) that the symmetric algebra construction $\Sym$ defines a symmetric monoidal functor going from cochain complexes equipped with a degree $1$ symmetric pairing (with monoidal structure the direct sum $\oplus$) to unital $\mathbb{P}_0$-algebras in cochain complexes (with monoidal structure the tensor product $\otimes$). We therefore have a locally constant $\Dun^1$-algebra in $\mathbb{P}_0$-algebras of classical observables
\[
\mathcal O^{c\ell}:(a,b)\longmapsto \Sym\big(\mathbb{V}(a,b)\big)\,.
\]

Forgetting the $1$-shifted bracket, $\mathcal O^{c\ell}$ defines a locally constant $\Dun^1$-algebra in commutative algebras (in $(\Cpx,\otimes)$) that is quasi-isomorphic to $\gamma^*\mathbb{K}[\mathbf q,\mathbf p]$ (the quasi-isomorphism is given by the map \eqref{eq:map-coho}), where $\mathbb{K}[\mathbf q,\mathbf p]$ is viewed as an $\Assoc$-algebra in commutative algebras with both products being the standard multiplication of polynomials. 

\subsection{Quantum observables: the odd laplacian}\label{ssec-3.4}

As explained in \cite[\S4.2]{CC}, quantizing constant $1$-shifted Poisson structures is rather easy, as the odd laplacian (that has nothing to do with our discrete Laplace operator) is somehow the pairing itself. In our situation it is given by \[
\Delta=\sum_{x\in\mathbb{Z}}\frac{\partial^2}{\partial \overline{\delta}_x\partial\delta_x}\,,
\]
acting on $\Sym\big(\mathbb{V}(a,b)\big)$. 
We actually don't need to know anything about quantization of $\mathbb{P}_0$-algebra, and just see the above as a deformation term for the differential on $\Sym\big(\mathbb{V}(a,b)\big)$: the new differential is $d_\hbar=d+\hbar\Delta$. 
\begin{rem}\label{rem-hbar}
    Here we have two possibilities. Either $\hbar\in\mathbb{K}^\times$ and we work over $\mathbb{K},$ or $\hbar$ is a formal parameter and we work over $\mathbb{K}[\hbar]$ or $\mathbb{K}[\![\hbar]\!]$. In what follows, we systematically extend scalars to the ring $\mathbf{K}$, which will be, depending on the context, $\mathbb{K}$, $\mathbb{K}[\hbar]$ or $\mathbb{K}[\![\hbar]\!]$. We let the reader figure out the details of each specific case. 
\end{rem}

We have the following properties (see \cite[\S4.2.2]{CC}): 
\begin{enumerate}
    \item The quantization of constant $1$-shifted Poisson structures is a symmetric monoidal functor from cochain complexes equipped with a degree $1$ symmetric pairing (with monoidal structure the direct sum) to $(\Cpx,\otimes)$. Hence the assignment 
    \[
    \mathcal O^q:(a,b)\longmapsto \Big(\Sym\big(\mathbb{V}(a,b)\big),d+\hbar\Delta\Big)
    \]
    defines a $\Dun^1$-algebra in $(\Cpx,\otimes)$. 
    \item The quantization of constant $1$-shifted Poisson structures preserves quasi-isomorphisms (this follows from a spectral sequence, or a deformation theoretic, argument). Therefore, $\mathcal O^q$ defines a locally constant $\Dun^1$-algebra in $(\Cpx,\otimes)$. 
\end{enumerate}

\begin{rem}
In fact, $\mathcal O^q$ defines a $\Dun^1$-algebra in $\mathbb{BD}_0$-algebras (following Costello--Gwilliam \cite{CG1}, as recalled in \cite[\S4.2]{CC})
\end{rem}

\subsection{Some computations: recovering the Weyl algebra}\label{ssec-3.5}

Observe that, since $\Delta$ vanishes on generators, then for every $x\in\mathbb{Z}$,  \[
d_\hbar\overline{\delta}_x=\delta_{x-1}-2\delta_x+\delta_{x+1}\,.
\]
As a consequence
\[
[\delta_{x+1}]=[2\delta_{x}-\delta_{x-1}]\qquad\textrm{and}\qquad 
[\delta_{x+1}-\delta_x]=[\delta_{x}-\delta_{x-1}]\,,
\]
where $[f]$ denotes the class of $f:\mathbb{Z}\cap(a,b)\to \mathbb{K}$ in $H^0\big(\mathcal O^q(a,b)\big)$. 

\medskip

Now recall the product $\star$ defined on $A=H^0\big(\mathcal O^q(J)\big)$, where $J=(-3,3)$, in Subsection \ref{ssec-2.1}: it is given as the composition 
\[
\begin{matrix}
A & \tilde\longleftarrow & H^0\big(\mathcal O^q(I_1)\big) &\\
\otimes& & \otimes & \underset{\mathrm{product}}{\overset{\mathrm{factorization}}{\relbar\joinrel\relbar\joinrel\relbar\joinrel\longrightarrow}} 
& H^0\big(\mathcal O^q(J)\big)=A\,,\\
A& \tilde\longleftarrow & H^0\big(\mathcal O^q(I_2)\big) & 
\end{matrix}
\]
where $I_1=(-2,1/2)$ and $I_2=(1/2,3)$. 

\begin{lem}
The following equality holds: $
[\delta_2-\delta_1]\star [\delta_0]-[\delta_0]\star[\delta_2-\delta_1]=\hbar
$. 
\end{lem}
\begin{proof}
On the one hand, 
\[
[\delta_0]\star[\delta_2-\delta_1]
=[\delta_0]\cdot[\delta_2-\delta_1]
=[\delta_0\cdot\delta_2-\delta_0\cdot\delta_1]\,.
\]
Actually, the above identity is even true at the cochain level (for functions with well-ordered disjoint supports, the factorization product is just the usual product). 

On the other hand, 
\begin{eqnarray*}
[\delta_2-\delta_1]\star[\delta_0] & = & 
[\delta_0-\delta_{-1}]\star[2\delta_1-\delta_2] \\
& = & [\delta_0-\delta_{-1}]\cdot[2\delta_1-\delta_2] \\
& = & [2\delta_0\cdot\delta_1-2\delta_{-1}\cdot\delta_1-\delta_0\cdot\delta_2+\delta_{-1}\cdot\delta_2]\,.
\end{eqnarray*}
The difference of these two gives 
\[
[\delta_2-\delta_1]\star[\delta_0]-[\delta_0]\star[\delta_2-\delta_1]=
[3\delta_0\cdot\delta_1
-2\delta_{-1}\cdot\delta_1
-2\delta_0\cdot\delta_2
+\delta_{-1}\cdot\delta_2]=\hbar\,, 
\]
because 
\[
3\delta_0\cdot\delta_1
-2\delta_{-1}\cdot\delta_1
-2\delta_0\cdot\delta_2
+\delta_{-1}\cdot\delta_2
=\hbar+d_\hbar(\overline{\delta}_1\cdot\delta_{-1}-\overline{\delta}_0\cdot\delta_0-2\overline{\delta}_1\cdot\delta_0)\,.
\]
Indeed: 
\begin{itemize}
    \item $d_\hbar(\overline{\delta}_1\cdot\delta_{-1})=\delta_0\cdot\delta_{-1}-2\delta_1\cdot\delta_{-1}+\delta_2\cdot\delta_{-1}$. 
    \item $d_\hbar(\overline{\delta}_0\cdot\delta_0)=\delta_{-1}\cdot\delta_0-2\delta_0\cdot\delta_0+\delta_1\cdot\delta_0+\hbar$. 
    \item $2d_\hbar(\overline{\delta}_1\cdot\delta_0)=2\delta_0\cdot\delta_0-4\delta_1\cdot\delta_0+2\delta_2\cdot\delta_0$. 
\end{itemize}
\end{proof}
As a consequence, the assignment $\mathbf{q}\mapsto [\delta_0]$, $\mathbf{p}\mapsto [\delta_2-\delta_1]$ defines a morphism 
\[
\faktor{\mathbf{K}\langle\mathbf{q},\mathbf{p}\rangle}{(\mathbf{p}\mathbf{q}-\mathbf{q}\mathbf{p}=\hbar)}\longrightarrow (A,\star)\,,
\]
that is actually an isomorphism; this last claim can be checked either by passing to the associated graded (when $\hbar$ is a scalar) or by sending $\hbar$ to $0$ (when it is a parameter). 

\subsection{Discrete time evolution automorphism}\label{ssec-3.6}

We consider the action of $\mathbb{Z}$ on $\Dun^1$ by translations (see Subsection \ref{ssec-2.2}), and observe that the $\Dun^1$-algebra $\mathcal O^q$ is actually $\mathbb{Z}$-equivariant: for every $n\in\mathbb{Z}$ and every open interval $I\subset \mathbb{R}$, we have an isomorphism 
\begin{eqnarray*}
\mathcal O^q(I) & \longrightarrow & \mathcal O^q(I+n) \\
f & \longmapsto & \big(x\mapsto 
f(x-n)\big)\,.
\end{eqnarray*}
In particular, $1\in\mathbb{Z}$ acts by sending $\delta_x$ to $\delta_{x+1}$ and $\overline{\delta}_x$ to $\overline{\delta}_{x+1}$. 

\medskip

Hence, the induced generating automorphism of $(A,\star)$ sends $\mathbf{q}$ to $\mathbf{q}+\mathbf{p}$, and $\mathbf{p}$ to itself. This coincides with the automorphism given by exponentiating the (inner) derivation $[\frac{1}{2\hbar}\mathbf{p}^2,-]$ (or, equivalently, conjugating with $e^{\frac{1}{2\hbar}\mathbf{p}^2}$), and thus gives back the automorphism of the continuous model. 

\subsection{Time reversal anti-involution and quantum states}\label{ssec-3.7}

Let us now consider the time reversal action of $C_2$ on $\Dun^1$ from Subsection \ref{ssec-2.2}. The $\Dun^1$-algebra $\mathcal O^q$ carries a $C_2$-equivariant structure: for every open interval $I\subset\mathbb{R}$ we have an isomorphism
\begin{eqnarray*}
\tau:\mathcal O^q(I) & \longrightarrow & O^q(-I) \\
f & \longmapsto & \big(x\mapsto 
f(-x)\big)\,.
\end{eqnarray*}
In particular, $\tau(\delta_x)=\delta_{-x}$ and $\tau(\overline{\delta}_x)=\overline{\delta}_{-x}$; therefore, 
\[
\tau([\delta_0])=[\delta_0]\quad\textrm{and}\quad\tau([\delta_2-\delta_1])=[\delta_{-2}-\delta_{-1}]=[\delta_1-\delta_2]\,.
\]
Hence the induced anti-involution on the Weyl algebra is given by $\mathbf{q}\mapsto \mathbf{q}$ and $\mathbf{p}\mapsto-\mathbf{p}$, and the left module of coinvariants is its quotient by the left ideal generated by $\mathbf{p}-(-\mathbf{p})=2\mathbf{p}$. We thus obtain $\mathbf{K}[\mathbf{q}]$ with left action given by 
\[
\mathbf{q}\cdot\mathbf{q}^n=\mathbf{q}^{n+1}\quad\textrm{and}\quad\mathbf{p}\cdot\mathbf{q}^n=n\mathbf{q}^{n-1}\,.
\]

\section{Discrete free massive scalar field in $1d$}\label{sec-massive}

Let $\alpha\in\mathbb{K}^\times$, and consider the endomorphism
\[
Q_\alpha:\Map(\mathbb{Z},\mathbb{K}) \longrightarrow \Map(\mathbb{Z},\mathbb{K})
\]
defined as follows: for every $h:\mathbb{Z}\to\mathbb{K}$ and every $x\in\mathbb{Z}$, 
\[
(Q_\alpha h)(x):=h(x-1)-(\alpha+\alpha^{-1})h(x)+h(x+1)\,.
\]
\begin{rem}\label{remark-mass}
Whenever $\mathbb{K}=\mathbb{R}$ and $\alpha>0$, $\alpha+\alpha^{-1}\geq 2$ could be thought as $2+m^2$, where $m=|\alpha^{1/2}-\alpha^{-1/2}|$ is the mass. Then we observe that $Q_\alpha=Q_1+m^2\mathrm{id}$, where $Q_1$ coincides with the discrete Laplace operator $Q$ from the previous section (massless case). 
\end{rem}

\subsection{Classical observables}\label{ssec-4.1}

Let $a<b$ be real numbers. If $f:\mathbb{Z}\to\mathbb{K}$ is supported in $(a,b)$ then $Q_\alpha f$ is supported in $(a-1,b+1)$; therefore $Q_\alpha$ defines a morphism
\[
\Map\big(\mathbb{Z}\cap(a+1,b-1),\mathbb{K}\big)\overset{Q_\alpha}{\longrightarrow}\Map\big(\mathbb{Z}\cap(a,b),\mathbb{K}\big)\,,
\]
whose cone we denote by $\mathbb{V}_\alpha(a,b)$. 

\medskip

We can upgrade the assignment $(a,b)\mapsto \mathbb{V}_\alpha(a,b)$ to a $\Dun^R$-algebra with values in $(\Cpx,\oplus)$ for every $R>0$: the operation
\[
\mathbb{V}_\alpha(I_1)\oplus\cdots\oplus\mathbb{V}_\alpha(I_n)\to \mathbb{V}_\alpha(J)
\]
is given by the sum, as in Subsection \ref{ssec-3.1}. 

\medskip

Applying the symmetric algebra functor $\Sym$, that is a symmetric monoidal functor from $(\Cpx,\oplus)$ to $\big(\mathtt{Alg}_{\mathbf{Com}}(\Cpx,\otimes),\otimes\big)$, we get a $\Dun^R$-algebra $\mathcal O_\alpha^{c\ell}:=\Sym(\mathbb{V}_\alpha)$ with values in differential graded commutative algebras. 

\begin{rem}[shifted pairing and shifted Poisson structure]\label{rem-shifted}
    In fact, $\mathbb{V}_\alpha(a,b)$ carries an additional structure: a degree $1$ symmetric pairing $\langle\!\langle\text{-,-}\rangle\!\rangle$. The definition is exactly the same as in Subsection \ref{ssec-3.3}, as well as the proof that it is compatible with the differential (just replace $2$ by $\alpha+\alpha^{-1}$) and the fact that it is compatible with the $\Dun^R$-algebra structure. As a consequence, $\mathcal O_\alpha^{c\ell}$ becomes a $\Dun^R$-algebra in $\mathbb{P}_0$-algebras. 
\end{rem}

\subsection{Cohomology}\label{ssec-4.2}

Everything works exactly as in Subsection \ref{ssec-3.2}. For every function $f:\mathbb{Z}\to\mathbb{K}$ having support in $[m,n]$, with $m,n\in\mathbb{Z}$, $(Q_\alpha f)(m-1)=f(m)$ and $(Q_\alpha f)(n+1)=f(n)$. 
Hence if $Q_\alpha f=0$ then $f(m)=f(n)=0$, and thus $f$ has support in $[m+1,n-1]$. Hence, by induction, if $f$ has finite support and $Q_\alpha f=0$, then $f=0$: in other words, $Q_\alpha$ is injective on finitely supported maps. 
Therefore, the cohomology of $\mathbb{V}_\alpha(a,b)$ is concentrated in degree $0$. By the rank theorem, if $|b-a|>2$ then $H^0\big((\mathbb{V}_\alpha(a,b)\big)$ is two dimensional. 

\medskip

Let us introduce four functions $u,v,A,B:\mathbb{Z}\to\mathbb{K}$ that will hopefully ease later calculations: 
\begin{itemize}
    \item $u(x):=\alpha^x$. 
    \item $v(x):=u(-x)=\alpha^{-x}$. 
    \item $A(x):=\displaystyle\frac{u(x)+v(x)}{2}$. 
    \item $B(x):=\displaystyle\frac{u(x)-v(x)}{\alpha-\alpha^{-1}}=
    \frac{x}{|x|}\Big(\alpha^{|x|-1}+\alpha^{|x|-3}+\cdots+\alpha^{3-|x|}+\alpha^{1-|x|}\Big)$. \\
    In particular, $B(x)=x$ whenever $\alpha=1$.   
\end{itemize}
Observe that all these functions belong to the kernel of $Q_\alpha$: for instance 
\[
(Q_\alpha u)(x)=\alpha^{x-1}-(\alpha+\alpha^{-1})\alpha^x+\alpha^{x+1}=0\,.
\]
We therefore have a morphism 
\begin{eqnarray*}
\varphi:H^0\big(\mathbb{V}_\alpha\big) & \longrightarrow & \mathbb{K}\mathbf{q}\oplus\mathbb{K}\mathbf{p}\\
{[f]} & \longmapsto & \underbrace{\sum_{x\in\mathbb{Z}}f(x)A(x)}_{=\langle\!\langle f,\overline{A}\rangle\!\rangle}\mathbf{q}+\sum_{x\in\mathbb{Z}}f(x)B(x)\mathbf{p}\,.
\end{eqnarray*}
This is well-defined thanks to $A,B\in\ker(Q_\alpha)$ and the compatibility between the pairing and the differential\footnote{Remember the notation from Subsection \ref{ssec-3.3}: an overlined symbol is used to denote the suspension of the element given by the original symbol. }: 
\begin{eqnarray*}
\langle\!\langle Q_\alpha g,\overline{A}\rangle\!\rangle
&=&\sum_{x\in\mathbb{Z}}\big(g(x-1)-(\alpha+\alpha^{-1})g(x)+g(x+1)\big)A(x)\\
&=&\sum_{x\in\mathbb{Z}}g(x)\big(A(x+1)-(\alpha+\alpha^{-1})A(x)+A(x-1)\big)\\
&=&\langle\!\langle\overline{g}, Q_\alpha  A\rangle\!\rangle=0\,.
\end{eqnarray*}
We let the reader check that $\varphi$ is an isomorphism whenever $|b-a|>2$. 

\medskip

As a consequence, we get that, as a $\Dun^1$-algebra taking values in $(\Cpx,\oplus)$ (or even complexes equipped with a degree $1$ symmetric pairing, see Remark \ref{rem-shifted}), $\mathbb{V}_\alpha$ is locally constant: the isomorphism $\varphi$ provides an equivalence $\mathbb{V}_\alpha\,\tilde\to\,\gamma^*(\mathbb{K}\mathbf{q}\oplus\mathbb{K}\mathbf{p})$. 

\medskip

This implies that $\mathcal O_\alpha^{c\ell}=\Sym(\mathbb{V}_\alpha)$ is also a locally constant $\Dun^1$-algebra (taking values in differential graded commutative algebras -- or even $\mathbb{P}_0$-algebras, see Remark \ref{rem-shifted}): $\varphi$ gives an equivalence $\mathcal O_\alpha^{c\ell}\,\tilde\to\,\gamma^*\mathbb{K}[\mathbf{q},\mathbf{p}]$. 

\subsection{Quantum observables}\label{ssec-4.3}

We now deform the differential of $\mathcal O_\alpha^{c\ell}$ by adding the odd laplacian 
\[
\Delta=\sum_{x\in\mathbb{Z}}\frac{\partial^2}{\partial\overline\delta_x\partial\delta_x}
\]
as in Subsection \ref{ssec-3.4}: the new differential is $d_\hbar=d+\hbar\Delta$, where: 
\begin{itemize}
\item The original differential $d$ is defined on generators by $d\overline{g}:=Q_\alpha g$ and extended by the Leibniz rule. 
\item The symbol $\hbar$ can either be understood as a non-zero scalar ($\hbar\in\mathbb{K}^\times$) or as a formal parameter, as explained in Remark \ref{rem-hbar}. 
\end{itemize}
We therefore define quantum observables $\mathcal O_\alpha^q:=\big(\Sym(\mathbb{V}_\alpha),d_\hbar\big)$. 

\medskip

The following properties can be checked in an \textit{ad hoc} way, and are particular instances of more general facts about the quantization of constant $1$-shifted Poisson structures (see e.g.~\cite[§4.2.2]{CC}, or Subsection \ref{ssec-3.4} above):
\begin{enumerate}
\item Quantum observables $\mathcal O_\alpha^q$ still form a $\Dun^1$-algebra, but with values in $(\Cpx,\otimes)$. Indeed, the new differential contains an order $2$ operator, that is no longer a derivation for the commutative product unless the functions that we multiply have disjoint supports. This can be checked by hand, or deduced from the functoriality of the quantization procedure explained in \cite[§4.2.2]{CC} (note that this quantization procedure actually takes values in $\mathbb{BD}_0$-algebras, hence so does $O_\alpha^q$. 
\item The $\Dun^1$-algebra is locally constant (using a spectral sequence argument, or a deformation theoretic one, this follows from that $\mathcal O_\alpha^{c\ell}$ is locally constant). 
\end{enumerate}

\medskip

We now recall from Subsection \ref{ssec-2.1} the description of an associative algebra $A_\alpha$ such that $\mathcal O_\alpha^q\simeq \gamma^* A_\alpha$: 
\begin{itemize}
    \item Let $A_\alpha:=H^0\big(\mathcal O_\alpha^q(J)\big)$, with $J:=(-4,4)$. 
    \item The product $\star$ on $A$ is defined as
    \[
    \begin{matrix}
    A_\alpha & \tilde\longleftarrow & H^0\big(\mathcal O_\alpha^q(I_1)\big) &\\
    \otimes& & \otimes & \underset{\mathrm{product}}{\overset{\mathrm{factorization}}{\relbar\joinrel\relbar\joinrel\relbar\joinrel\longrightarrow}} 
    & H^0\big(\mathcal O_\alpha^q(J)\big)=A_\alpha\,,\\
    A_\alpha& \tilde\longleftarrow & H^0\big(\mathcal O_\alpha^q(I_2)\big) & 
    \end{matrix}
    \]
    for any $I_1,I_2\subset J$ such that $|I_1|>2$, $|I_2|>2$ and $I_1<I_2$. 
\end{itemize}
\begin{lem}
    Let $\mathrm{q}:=[\delta_0]$ and $\mathrm{p}:=\frac12[\delta_1-\delta_{-1}]$. Then $\mathrm{p}\star\mathrm{q}-\mathrm{q}\star\mathrm{p}=\hbar$. 
\end{lem}
\begin{rem}
Observe that $\varphi(\mathrm{q})=\mathbf{q}$ and $\varphi(\mathrm{p})=\mathbf{p}$. 
\end{rem}
\begin{proof}
First observe that 
\begin{eqnarray*}
\delta_0-d_\hbar\big(\overline\delta_1+(\alpha+\alpha^{-1})\overline\delta_2\big) & = & \delta_0-Q_\alpha\big(\delta_1+(\alpha+\alpha^{-1})\delta_2\big) \\
& = & \big((\alpha+\alpha^{-1})^2-1)\delta_2-(\alpha+\alpha^{-1})\delta_3
\end{eqnarray*}
is supported in $\{2,3\}$ (we have used that $\Delta$ sends generators to $0$). 
Therefore, by picking $I_1=(-2,3/2)$ and $I_2=(3/2,4)$, we obtain 
\begin{eqnarray*}
[\delta_1-\delta_{-1}]\star[\delta_0] 
& = &
[\delta_1-\delta_{-1}]\star\left[\delta_0-d_\hbar\big(\overline\delta_1+(\alpha+\alpha^{-1})\overline\delta_2\big)
\right]\\
& = &
\left[(\delta_1-\delta_{-1})\cdot\Big(\delta_0-d_\hbar\big(\overline\delta_1+(\alpha+\alpha^{-1})\overline\delta_2\big)\Big)\right]
\end{eqnarray*}
Similarly, $\delta_0-d_\hbar\big(\overline\delta_{-1}+(\alpha+\alpha^{-1})\overline\delta_{-2}\big)$ is supported in $\{-3,-2\}$. Thus, picking $I_1=(-4,-3/2)$ and $I_2=(-3/2,2)$, we get
\[
[\delta_0]\star [\delta_1-\delta_{-1}]=\left[\Big(\delta_0-d_\hbar\big(\overline\delta_{-1}+(\alpha+\alpha^{-1})\overline\delta_{-2}\big)\Big)\cdot(\delta_1-\delta_{-1})\right]\,.
\]
We now compute the commutator: 
 \begin{eqnarray*}
[\delta_1-\delta_{-1}]\star[\delta_0]-[\delta_0]\star [\delta_1-\delta_{-1}] 
& = & 
\Big[d_\hbar\big((\alpha+\alpha^{-1})\overline\delta_{-2}+\overline\delta_{-1}-\overline\delta_1-(\alpha+\alpha^{-1})\overline\delta_2\big)\cdot(\delta_1-\delta_{-1})\Big] \\
& = & 
-\hbar\left\langle\!\left\langle(\alpha+\alpha^{-1})\overline\delta_{-2}+\overline\delta_{-1}-\overline\delta_1-(\alpha+\alpha^{-1})\overline\delta_2\,,\,\delta_1-\delta_{-1}\right\rangle\!\right\rangle \\
& = & 2\hbar\,.
\end{eqnarray*}
Here we have used the following general fact: for finitely supported $f,g:\mathbb{Z}\to\mathbb{K}$, $d_\hbar(\overline{f}\cdot g)=d_\hbar(\overline{f})g+\hbar\langle\!\langle\overline{f},g\rangle\!\rangle$ and thus $\big[d_\hbar(\overline{f})g\big]=-\hbar\langle\!\langle\overline{f},g\rangle\!\rangle$. 
\end{proof}
As a consequence of the above Lemma, we obtain an algebra morphism 
\[
\faktor{\mathbf{K}\langle\mathbf{q},\mathbf{p}\rangle}{(\mathbf{p}\mathbf{q}-\mathbf{q}\mathbf{p}=\hbar)}\longrightarrow (A_\alpha,\star)\,,
\]
that can be shown to be an isomorphism: again, this can be checked by passing to associated graded (whenever $\hbar$ is a scalar) or by reasoning modulo $\hbar$ (whenever $\hbar$ is a variable). 

\medskip

It is remarkable to notice that our algebra of quantum observables in fact doesn't depend on the mass. We will see in the next Subsection that the mass ``matters back'' when time evolves (in a similar way as in \cite[Ch.~4 Sect.~3]{CG1}). 

\subsection{Time evolution}\label{ssec-4.4}

As in Subsection \ref{ssec-3.6}, we consider the action of $\mathbb{Z}$ on the $\Dun^1$-algebra $\mathcal O_\alpha^q$: for every $n\in\mathbb{Z}$ and every open interval $I\subset \mathbb{R}$, we have an isomorphism 
\begin{eqnarray*}
\mathcal O^q(I) & \longrightarrow & \mathcal O^q(I+n) \\
f & \longmapsto & \big(x\mapsto 
f(x-n)\big)\,.
\end{eqnarray*}
In particular, $1\in\mathbb{Z}$ acts by sending $\delta_x$ to $\delta_{x+1}$ and $\overline{\delta}_x$ to $\overline{\delta}_{x+1}$. 

\begin{lem}\label{lemma-automorphism}
The induced automorphism of the Weyl algebra \[
\faktor{\mathbf{K}\langle\mathbf{q},\mathbf{p}\rangle}{(\mathbf{p}\mathbf{q}-\mathbf{q}\mathbf{p}=\hbar)}
\]
is given by 
\[
\mathbf{q}\longmapsto \frac{\alpha+\alpha^{-1}}{2}\mathbf{q}+\mathbf{p}
\quad\textrm{and}\quad\mathbf{p}\longmapsto\left(\frac{\alpha-\alpha^{-1}}{2}\right)^2\mathbf{q}+\frac{\alpha+\alpha^{-1}}{2}\mathbf{p}\,.
\]
\end{lem}
\begin{proof}
The automorphism sends $\mathrm{q}=[\delta_0]$ to $[\delta_1]$ and $\mathrm{p}=\frac12[\delta_1-\delta_{-1}]$ to $\frac12[\delta_2-\delta_0]$. 
Recall that on generators both differentials (the original and the deformed one) coincide, hence we can use $\varphi$ to compute what their classes gives in the Weyl algebra: 
\[
\varphi([\delta_1])=\frac{\alpha+\alpha^{-1}}{2}\mathbf{q}+\mathbf{p}
\quad
\textrm{and}
\quad
\varphi([\delta_2-\delta_0])=\frac{(\alpha-\alpha^{-1})^2}{2}\mathbf{q}+(\alpha+\alpha^{-1})\mathbf{p}\,.
\]
\end{proof}

\begin{rem}
According to \cite[Ch.~4 Sect.~3]{CG1} the (inner) derivation given by infinitesimal translation of the continuous model is $\frac{1}{2\hbar}[\mathbf{p}^2-m^2\mathbf{q}^2,-]$, sending $\mathbf{q}$ to $\mathbf{p}$ and $\mathbf{p}$ to $m^2\mathbf{q}$. A simple calculation shows that, after exponentiation, this gives the following automorphism of the Weyl algebra: 
\[
\mathbf{q}\longmapsto \cosh(m)\mathbf{q}+\frac{\sinh(m)}{m}\mathbf{p}
\quad\textrm{and}\quad\mathbf{p}\longmapsto m\sinh(m)\mathbf{q}+\cosh(m)\mathbf{p}\,.
\]
This doesn't quite coincide with the automorphism given by the discrete model (see Lemma \ref{lemma-automorphism}). 
However, in the context of Remark \ref{remark-mass} and very informally, for 
$\alpha\to 1^+$, 
\[
m=|\alpha^{1/2}-\alpha^{-1/2}|\sim\log(\alpha)\quad\textrm{and}\quad m\sim\sinh(m)\,.
\]
Therefore, 
\[
\cosh(m)\sim {\frac{\alpha+\alpha^{-1}}{2}}\,, \quad
{\frac{\sinh(m)}{m}}\sim 1\quad\textrm{and}\quad 
m\sinh(m)\sim {\left(\frac{\alpha-\alpha^{-1}}{2}\right)^2}\,.
\]
This means that the discrete time evolution is ``close'' to the continuous one whenever $\alpha\to 1^+$. 

Nevertheless, unlike with the algebra of observables (that we obtain without taking a continuum limit), it seems that if one wants to recover the time evolution of the continuous model then one should probably take an appropriate continuum limit. 
\end{rem}

\subsection{Fock module}\label{ssec-4.5}

The $C_2$-equivariant structure on the $\Dun^1$-algebra $\mathcal O_\alpha^q$ is exactly the same as the one in Subsection \ref{ssec-3.7}: for every open interval $I\subset\mathbb{R}$ we have an isomorphism
\begin{eqnarray*}
\tau:\mathcal O_\alpha^q(I) & \longrightarrow & O_\alpha^q(-I) \\
f & \longmapsto & \big(x\mapsto 
f(-x)\big)\,.
\end{eqnarray*}

\medskip

The induced anti-involution of $(A,\star)$, that we still denote $\tau$, therefore satisfies $\tau(\mathrm{q})=\mathrm{q}$ and $\tau(\mathrm{p})=-\mathrm{p}$. Hence, through the isomorphism constructed in Subsection \ref{ssec-4.3}, we get the anti-involution of the Weyl algebra defined by $\mathbf{q}\mapsto\mathbf{q}$ and $\mathbf{p}\mapsto-\mathbf{p}$. 

\medskip

The associated coinvariant left module of the Weyl algebra, that is the quotient of the Weyl algebra by the left ideal generated by $\mathbf{p}$, is $\mathbf{K}[\mathbf{q}]$ equipped with the following module structure: for every $n\in\mathbb{N}$, 
\[
\mathbf{q}\cdot\mathbf{q}^n=\mathbf{q}^{n+1}
\quad\textrm{and}\quad
\mathbf{p}\cdot\mathbf{q}^n=n\mathbf{q}^{n-1}\,.
\]

\end{document}